\def\draft{0}  % 1 to include author notes

\documentclass[11pt]{article}

\usepackage[letterpaper,hmargin=1in,vmargin=1.25in]{geometry}
\usepackage{amsfonts,url,ifthen,epsfig,eufrak,amsmath,amssymb,hyperref}

\ifnum\draft=1
\newcommand{\Rnote}[1]{{\bf [Ronen's Note: #1]}}
\newcommand{\Onote}[1]{{\bf [Omer's Note: #1]}}
\else
\newcommand{\Rnote}[1]{}
\newcommand{\Onote}[1]{}
\fi

\newcommand{\set}[1]{\left\{ #1 \right\}}
\newtheorem{theorem}{Theorem}[section]
\newtheorem{definition}[theorem]{Definition}
\newtheorem{lemma}[theorem]{Lemma}
\newtheorem{proposition}[theorem]{Proposition}

\newtheorem{remk}[theorem]{Remark}
\newtheorem{examp}[theorem]{Example}

\def\FullBox{\hbox{\vrule width 8pt height 8pt depth 0pt}}

\def\qed{\ifmmode\qquad\FullBox\else{\unskip\nobreak\hfil
\penalty50\hskip1em\null\nobreak\hfil\FullBox
\parfillskip=0pt\finalhyphendemerits=0\endgraf}\fi}

\def\qedsketch{\ifmmode\Box\else{\unskip\nobreak\hfil
\penalty50\hskip1em\null\nobreak\hfil$\Box$
\parfillskip=0pt\finalhyphendemerits=0\endgraf}\fi}

\newenvironment{proof}{\begin{trivlist} \item {\bf Proof:~~}}
  {\qed\end{trivlist}}

\newcommand{\eqdef}{\mathbin{\stackrel{\rm def}{=}}}
\newcommand{\R}{{\mathbb R}}
\newcommand{\N}{{\mathbb{N}}}

\newcommand{\zo}{\{0,1\}}

\newcommand{\pr}[1]{\Pr\left[#1\right]}

\newcommand{\ip}[2]{\left<#1,#2\right>}
\newcommand{\norm}[1]{\left\|#1\right\|}

\newcommand{\E}{\mathop{\mathrm E}\displaylimits}
\newcommand{\Var}{\mathop{\mathrm{Var}}\displaylimits}

\newcommand{\supp}{\mathrm{supp}}

\newcommand{\majp}{\mathrm{Maj}_p}
\newcommand{\omu}{\overline{\mu}}
\newcommand{\ef}{\mathcal{E}}
\newcommand{\abs}[1]{| #1 |}

\title{The Player's Effect
\ifnum\draft=1{\\ \small \sc Working Draft, Please Do Not Distribute }\fi }

\author{
Ronen Gradwohl \thanks{Department of Computer Science and Applied
Mathematics, The Weizmann Institute of Science, Rehovot, 76100
Israel. E-mail: \texttt{ronen.gradwohl@weizmann.ac.il}.} \and Omer
Reingold~\thanks{Department of Computer Science and Applied
Mathematics, The Weizmann Institute of Science, Rehovot, 76100
Israel. E-mail: \texttt{omer.reingold@weizmann.ac.il}. Research
supported by US-Israel Binational Science Foundation Grants
2002246 and 2006060.} \and
Ariel Yadin \thanks{Department of Computer
Science and Applied Mathematics, The Weizmann Institute of
Science, Rehovot, 76100 Israel. E-mail:
\texttt{ariel.yadin@weizmann.ac.il}.} \and
Amir Yehudayoff
\thanks{Department of Computer Science and Applied Mathematics,
The Weizmann Institute of Science, Rehovot, 76100 Israel. E-mail:
\texttt{amir.yehudayoff@weizmann.ac.il}.
Research supported by a grant from the Israel Ministry of Science (IMOS) - Eshkol
Fellowship.}
 }

\begin{document}
%\begin{titlepage}
\date{}

\maketitle

\thispagestyle{empty}

\begin{abstract}
In a function that takes its inputs from various players,
the effect of a player measures the variation he can cause in the expectation of that function.
In this paper we prove a tight upper bound on the number
of players with large effect, a bound that holds even when the players' inputs are only
known to be pairwise independent. We also study the effect of a set of players, and show that there always exists
a ``small'' set that, when eliminated, leaves every set with little effect. Finally, we ask whether
there always exists a player with positive effect. We answer this question differently
in various scenarios, depending on the properties of the function and the distribution of
players' inputs. More specifically, we show that if the function is non-monotone or the distribution
is only known to be pairwise independent, then it is possible that all players have 0 effect. If
the distribution is pairwise independent with minimal support, on the other hand, then
there must exist a player with ``large'' effect.
\end{abstract}

\section{Introduction}
A general recurring theme in the analysis of games is the juxtaposition of two distinct sources
of players' motivation to act strategically: The first is the myopic maximization of their own immediate
gain, and the second is a consideration of the effect their behavior has on
other players and possibly a collective outcome. This theme surfaces in many settings. For example,
in a repeated game in which players both maximize their utilities and learn others' preferences,
players must strike a balance between playing to obtain immediate gain and playing to learn
(or teach). Another example is the setting of an extensive game with many players, in which the
actions of a player yield him some utility but also affect the choices of the subsequent players.

In such settings, players who have little impact on others or on a collective outcome more or less
ignore the second source of motivation for their strategic behavior, and can be shown to act myopically. This has
been demonstrated more precisely for many examples: The provision of a public good \cite{MP90,AS00},
repeated games \cite{S90,AS01},
and mechanisms for choosing equilibria in private information economies \cite{AS07}.

There are many ways one can quantify the impact of a player on others or on a collective outcome.
Two notable notions that have been widely studied in the economics and computer science
literatures respectively are the {\em effect} and the {\em influence} of players. We illuminate
the distinction between the two notions via the example of voting.

Suppose there are $n$ players, and let a function $f$
be a voting scheme between two candidates.
Each player has a signal $X_i$, a binary random variable, and given $n$ binary inputs
the voting scheme $f$ outputs the name of one of the candidates.
The effect of a player is the amount of variation he can cause in the expectation
(over players' signals) of $f$ by a unilateral change in his own signal.
Note
that this is an a priori notion -- a player's effect is measured without assuming any
knowledge of the other players' votes, only their distribution. The influence of a player,
on the other hand, is defined as
the probability (over all the players' votes) that a specific player casts the deciding vote.
This is an a posterior notion, since a player has impact on the outcome after others have already
received their signals and is conditional on their respective votes.

\paragraph{Previous Work}
The notion of influence was introduced by Ben-Or and Linial \cite{BL89}.
The seminal paper in this
line of work is that of Kahn, Kalai and Linial \cite{KKL88} (henceforth KKL), in which they showed that
in any voting scheme, if the players' signals are independent then there always
exists a player with ``large'' influence. Following this paper, there has been much
work studying the notion of influence (see Kalai and Safra \cite{KS05} for a survey).

The notion of effect was studied by Malaith and Postlewaite \cite{MP90} and by
Fudenberg, Levine, and Pesendorfer \cite{FLP98}. The results most similar to ours are those of
Al-Najjar and Smorodinsky \cite{AS00}, who
gave tight bounds on the number of players with large effect. One of their assumptions is that
the players' signals are independent (or at least independent conditional on some outside
information). Their methods do
not apply to general distributions.

Haggstrom, Kalai and Mossel \cite{HKM06}
studied the notion of effect in general distributions, and showed that there is complete aggregation of
information (this means that a small tendency towards one outcome
for each player gives a strong tendency in the general outcome)
under certain conditions related to the distribution and the effects.

\paragraph{Our Results} We answer questions similar to those of \cite{AS00} and \cite{KKL88}
in the context of distributions that are only weakly independent.
As in \cite{AS00}, we give a tight bound on the number of players with large effect.
The novelty of our bound
is that we assume only minimal independence of players' signals. More precisely, our bound
holds even when the players are only pairwise independent. Note that when the players are 1-wise
independent, all players may have maximal effect. Additionally, we study the effects of
coalitions of players. We show that a small set of players can be eliminated, leaving only
coalitions with small effect.

We also ask whether a KKL-type theorem holds for effect -- that is, does there always exist some
player with large effect? We have three results here: First, we observe that if the function is not
monotone, then it is possible that all effects are 0, even in the fully independent case.
Second, we show that there exists a pairwise independent distribution and a monotone function
such that all players' effects are again 0. Also, we use similar ideas to show that there
exists a pairwise independent distribution
and a monotone function such that all players have influence 0; i.e. a KKL-type theorem does not
hold for influence either, in the case of pairwise independence.
Finally, we give a positive result: we show that
if the distribution is pairwise independent and of minimal size, then there exists a player
with very large effect.

\paragraph{Organization} The rest of this paper is organized as follows. Section~\ref{sec:def-and-main}
begins with some formal definitions, and then proceeds with formal statements of all our results.
Section~\ref{sec:num-pivotal} contains the proofs of our theorems bounding the number of
players and sets with large effect, and Section~\ref{sec:kkl-type-results} contains the proofs
of our results on KKL-type theorems.

\section{Definitions and Main Results}
\label{sec:def-and-main}
\subsection{Definitions}
Let $n\in\N$, and let $S$ be a finite set (whose size may depend on $n$).
Let $f:S^n\mapsto [-1,1]$ be some
function, let $X_1,\ldots,X_n$ be $n$ random variables (which
are not necessarily independent) taking values in $S$, and denote by $X=(X_1,\ldots,X_n)$. We think of each $X_i$ as
a player. If $S=\{0,1\}$, then we have the notion of the effect of
a player.

\begin{definition}[Effect]
For $i \in [n]$, denote $$\ef_i(f,X) =
\big|\E_{X}[f|X_i=1]-\E_X[f|X_i=0]\big|$$ the effect of player
$X_i$ in $f$ with respect to $X$. For $\alpha \in \R$, we say that
$X_i$ has {\em effect} $\alpha$ in $f$ with respect to $X$ if
$$\ef_i(f,X) > \alpha.$$
\end{definition}
Denote by $K(f,X,\alpha)$ the number of players with effect
$\alpha$ in $f$ with respect to $X$.

For arbitrary discrete sets $S$, we have the following generalization of the effect of a player.

\begin{definition}[pivotal player]
\label{def:pivotal-player}
For $p,\alpha \in \R$, we say that $X_i$ is {\em
$(p,\alpha)$-pivotal} in $f$ with respect to $X$ if
$$\Pr_{X_i}\Big[\big|\E_X[f|X_i]-\E_X[f]\big|>\alpha\Big]>p.$$
\end{definition}
Denote by $K(f,X,p,\alpha)$ the number of $(p,\alpha)$-pivotal
players in $f$ with respect to $X$.

Additionally, for a random variable $X=(X_1,\ldots,X_n)$ and a set $T\subseteq [n]$,
let $X_T=\left(X_i\right)_{i\in T}$ be the projection
of $X$ onto the variables in $T$.

\begin{definition}[Pivotal set of players]
\label{def:pivotal-set}
For $p,\alpha \in \R$, we say that $T\subseteq [n]$ is {\em
$(p,\alpha)$-pivotal} in $f$ with respect to the random variable $X=(X_1,\ldots,X_n)$ if
$$\Pr\left[\Big|\E[f|X_T]-\E[f]\Big|>\alpha\right]>p.$$
\end{definition}

The notions of effect, pivotal players, and pivotal sets of players are
relevant even for variables that are not fully independent.
We relax the assumption of full independence as follows.

\begin{definition}[$k$-wise independence]
\label{def:k-wise}
The random variables $X_1,\ldots,X_n$ are {\em $k$-wise independent} if
for any subset $T\subset[n]$, $|T|\leq k$, the random variables
$\{X_i:i\in T\}$ are independent.
\end{definition}

We also state here the precise definition of influence.
\begin{definition}[Influence]
\label{def:influence}
Let $f:\{0,1\}^n \to \{0,1\}$ be a function, and let $\mu$ be a distribution on $\{0,1\}^n$.
The \emph{influence} of the $i$'th player is
$$ I_i(f,\mu) = \Pr_{x \sim \mu} \left[  f(x) \neq f(x \oplus e_i) \right] , $$
where $e_i$ is the vector with $1$ at the $i$'th index and $0$ elsewhere, and $\oplus$ is bitwise XOR.
\end{definition}

\subsection{Main Results}

The following theorem bounds the number of pivotal players.

\begin{theorem}
\label{thm:general} Let $n\in\N$, and let $S$ be a finite set.
Let $f:S^n\mapsto [-1,1]$ be some function. Let $X_1,\ldots,X_n$
be $n$ pairwise independent random variables taking values in $S$, and denote $X=(X_1,\ldots,X_n)$.
Then for every positive $\alpha,p \in \R$,
$$K(f,X,p,\alpha)<\frac{8}{p\alpha^2}.$$
\end{theorem}

We note that the above theorem was known for the case of fully
independent random variables (see \cite{AS00}). However, we
consider the much more general case of pairwise independence. We
note also that for the case where all the players vote the same (i.e. the signals are 1-wise independent),
and $f$ is boolean such that $f(0,\ldots,0) = 0$ and
$f(1,\ldots,1) = 1$,
$$K(f,X,1,1) = n,$$ so the conclusion of the theorem does not hold.

We now turn our attention to the effect of a set of players -- that is, how much can
a set of players cause variation from the expectation of a function. A first observation
is that if some player $i$ is $(p,\alpha)$-pivotal, then any set of players that contains
$i$ is at least as pivotal.

However, we still prove a rather strong statement. The theorem is a generalization
of a theorem of Gradwohl and Reingold \cite{GR07} to the case of variables that are
not fully independent. Roughly, the theorem states that it is
possible to eliminate some not too large set of players $T$, so that \textbf{any} set
of some bounded size that does not intersect $T$ will have little influence.

\begin{theorem}
\label{thm:ave-case-sets}
Fix some natural number $m$. Then for any set of $2m$-wise independent random variables
$X_1,\ldots,X_n$, any $0<\alpha<1$, $0<p<1$, and any function $f$, the following holds:
there exists a set $C \subseteq [n]$ of size $|C|\leq 8m / p\alpha^2$, such that
for \textbf{all} $T \subseteq [n]\setminus C$ of size $|T| \leq m$, the set $T$ is \textbf{not}
$(p,\alpha)$-pivotal.
\end{theorem}

\subsection{Example}
\label{sec:lower-bounds}

The following example, a function we call $\majp$, is from \cite{AS00}, and shows that our
bound on the number of pivotal players in Theorem~\ref{thm:general} is tight. Let $n \in \N$, let $0<p<1$,
and let $S = \set{0,1,\perp}$. For each player $i \in [n]$,
suppose $\pr{X_i=0}=\pr{X_i=1}=p/2$, and $\pr{X_i=\perp}=1-p$. Let
$f:S^n\mapsto \{0,1\}$ be the majority function over
all players that did not output $\perp$.

We say that a player \emph{participates} if he does not output $\perp$.
Then every player participates with probability $p$,
and is influential if the remaining players who did not output $\perp$ are split evenly
between 0's and 1's.
If the number of participating players is $pn$
(which is roughly the case with overwhelming probability),
then the player is influential with probability roughly $1/\sqrt{pn}$. Thus,
every one of the $n$ players is roughly $(p,1/\sqrt{pn})$-pivotal.
Setting $\alpha \approx 1/\sqrt{pn}$, we get that $n\approx
\frac{1}{p\alpha^2}$, which is the number of $(p,\alpha)$-pivotal players.
Note that we can vary the value of $\alpha$ in this example by picking
some natural number $k\leq n$ such that $\alpha \approx 1/\sqrt{pk}$. The function
we consider is then the function above, but only over some arbitrary set of
$k$ players. In this case, those $k$ players will all be $(p,\alpha)$-pivotal.

\subsection{General KKL-Type Results}

The questions raised in this section are motivated by the celebrated result of Kahn, Kalai, and Linial \cite{KKL88}.
Roughly, their result states that for every balanced Boolean function on
$\set{0,1}^n$ and fully independent inputs there exists a player with influence at least
$\Omega(\log(n) / n)$.

The notions of influence and effect are closely related in several specific cases.
Here we ask whether a KKL-type theorem holds with regard to effect. More specifically, we ask the following question:
\begin{center}
\textbf{Does there always exist a player with large effect?}
\end{center}
We show that the answer to this question depends strongly on the
underlying distribution of the players and the
properties of the function.  We also extend one of our negative results to the original notion of influence, and show that
a KKL-type theorem does not hold for general distributions, even for monotone functions (see
Section~\ref{scn: negative result for influence}).

\subsubsection{Full Independence}
We first consider the question stated above for the case in which the players' signals are fully independent.
The first observation is that for monotone functions, the notions of influence and
effect are equivalent \cite{HKM06}.
This means that for fully independent players and balanced monotone functions, the original KKL theorem
roughly states that there exists a player with effect at least $\log(n)/n$.

How about non-monotone functions? It is possible to transform a non-monotone function into a monotone
one in such a way that the influences do not increase \cite{KKL88}. Hence, a lower bound on the
influence of a player for all monotone functions gives a similar bound for all functions.
However, such a transformation does not exist for effects. Consider, for example, the PARITY function,
in which each player independently outputs a bit generated by a fair coin toss. Then here
all influences are 1, but all effects are 0. In particular, a KKL-type theorem for effect does not hold for
non-monotone functions -- there is no non-trivial lower bound on the effect of a player
in a non-monotone function,
even in the case of full independence. Hence, in the following sections, we only
consider monotone functions.

\subsubsection{Pairwise Independence -- Negative Results}
In the previous section we noted that for monotone functions with full independence, the original
KKL theorem states that there exists a player with large effect. We also saw that without monotonicity,
this does not hold. In this section we ask whether full independence is necessary (when
the function is monotone), or whether some weaker guarantee such as pairwise independence suffices.

In Section~\ref{scn: negative result for effect} we show that there exists
a balanced monotone function $f$ and a pairwise independent distribution $D$ over
$\zo^n$ such that
$$\ef_i(f,D) = 0$$ for all $i \in [n]$.
This implies that there is no non-trivial lower bound on the
effect of a player for pairwise independent distributions, even for monotone functions.

Furthermore, in Section~\ref{scn: negative result for influence} we extend these results to
show that there exists a balanced monotone function $g$ and a pairwise independent distribution $D'$ such that
$$I_i(g,D') = 0$$ for all $i \in [n]$.

\subsubsection{Pairwise Independence -- Positive Result}
In the previous section, we showed that, for the case of pairwise independence, monotonicity
does not suffice in order for some KKL-type theorem to hold. In this section, we show that
a KKL-type theorem does hold in a restricted special case.
Roughly, we show that
for all pairwise independent distributions with minimal support size, there is a player
with effect at least $1/\sqrt{n}$ (for any balanced function).

\begin{theorem} \label{thm: p.w. positive}
Let $n +1 = 2^k$.
Let $\mu$ be a pairwise independent distribution on $(X_1,\ldots,X_n)\in\set{0,1}^n$, with
marginals $1/2$ and $|\supp(\mu)|=n+1$. Let $f:\set{0,1}^n\mapsto\set{0,1}$.
Then $$\sum_{i\in[n]}\left(\ef_i(f)\right)^2 = \frac{\Var[f]}{4}.$$
\end{theorem}

A function $f$ is called balanced if its expectation is, say, $1/2$, which also implies that
its variance is $1/4$. For such functions, the above theorem states that
the sum of squares of effects is at least $1/16$, and so there exists a player with effect
at least $1/(4\sqrt{n})$.

We now discuss the premise of the above theorem, namely pairwise independent distributions on $\set{0,1}^n$
with support size $n+1$. First, we note that there are such distributions -- $\mu$ and $\omu$, constructed
in Section~\ref{scn: negative result for effect}
are examples. Second, any pairwise independent distribution on $\set{0,1}^n$ has
support of size {\em at least} $n+1$. Finally, such a distribution with support of size {\em exactly}
$n+1$ must be uniform on its support -- see Benjamini, Gurel-Gurevich, and Peled \cite{BGP07}.

We also note that the distribution $D$ from the previous section that served as our counter-example
to any KKL-type theorem for effect has support of size $2(n+1)$. It seems a small difference in support size can
make a significant difference: For any pairwise independent distribution with support size $n+1$ and
balanced function, there exists a player with effect roughly $1/\sqrt{n}$. On the other hand, a convex sum
of two such distributions yields a distribution for which there exists a balanced function in which all the
effects are 0.

\subsection{Preliminaries}
We need some preliminary definitions. Let $X$ be a random variable
taking values in $\set{0,1}^k$. For two functions
$g,h:\{0,1\}^k\mapsto [-1,1]$, define the inner product (with
respect to $X$)
$$\ip{g}{h}=\sum_{x\in\{0,1\}^k}\pr{X=x} \cdot g(x)\cdot h(x)$$ (for
simplicity of notation, we omit the dependency on $X$ from the
notation, and will make sure it is clear from the context). Define
the norm of $g$ to be
$$\norm{g} = \sqrt{\ip{g}{g}}.$$

\section{The Number of Pivotal Players and Sets}
\label{sec:num-pivotal}
We begin with a weaker result then our main theorem because its proof is
instructive in that it contains many of the ideas of the main theorem.
The more general case will be proven in Section~\ref{sec:more-general}.

\subsection{Warm-Up: Binary Independent Inputs}
The theorem we prove here is the following:
\begin{proposition}
Let $n \in \N$, and let $f:\{0,1\}^n\mapsto [-1,1]$. If
$X_1,\ldots,X_n$ are $n$ fully independent random variables such that
$\Pr[X_i = 1] = \Pr[X_i = 0] = 1/2$ for all $i \in [n]$, then for every positive $\alpha
\in \R$,
$$K(f,X,\alpha)<\frac{4}{\alpha^2},$$
where $X=(X_1,\ldots,X_n)$.
\end{proposition}

\begin{proof}
Let $k=K(f,X,\alpha)$, and without loss of generality assume
that the first $k$ variables are the ones with effect $\alpha$. Define the function
$f':\{0,1\}^k\mapsto [-1,1]$ as
$$f'(x_1,\ldots,x_k)=\E_X[f|X_1=x_1,\ldots,X_k=x_k].$$ Also denote
$X' = (X_1,\ldots,X_k)$. For every $b \in \set{0,1}$ and for
every $i \in [k]$,
$$\E_{X'}[f'|X_i=b] = \E_{X'}[\E_X[f|X_1=x_1,\ldots,X_k=x_k]  |X_i=b] =
\E_X[f|X_i = b],$$ and thus player $i$ has effect $\alpha$ in $f'$ if and
only if player $i$ has effect $\alpha$ in $f$.
Furthermore, assume without loss of generality that for all $i
 \in [k]$, we have $\E_X[f|X_i=0]>\E_X[f|X_i=1]$, and so
$$\E_X[f|X_i=0]-\E_X[f|X_i=1]>\alpha$$ (otherwise consider the function
$g(x_1,\ldots,x_n) = f(x_1,\ldots,x_{i-1}, 1-x_i,\ldots,x_n)$,
and note that this does not alter the effects of the players).

For $i \in [k]$, define the functions $b_i:\{0,1\}^k\mapsto
[-1,1]$ as
$$\forall \ x = (x_1,\ldots,x_k) \in \set{0,1}^k \ \ b_i(x)=(-1)^{x_i}.$$

The functions $b_i$ are useful because
$$\ip{b_i}{f'}= 2^{-k} \sum_{x:x_i=0}f'(x)- 2^{-k} \sum_{x:x_i=1}f'(x)=
\frac{1}{2}\E_X[f'|X_i=0]-\frac{1}{2}\E_X[f'|X_i=1]>\frac{\alpha}{2}$$
by assumption on the effects of $X_1,\ldots,X_k$. Also for $i,j
\in [k]$, since $X_1,\ldots,X_n$ are independent, $$\ip{b_i}{b_j}
= \Pr_{X'}[X_i = X_j] - \Pr_{X'}[X_i \neq X_j] = \left\{
\begin{array}{cc} 1 & i = j
\\ 0 & i \neq j,
\end{array} \right.$$ and so $$\norm{b_1+\ldots+b_k} = \sqrt{k}.$$
Now, on one hand,
$$\ip{b_1+\ldots+b_k}{f'}>\frac{k\alpha}{2}.$$

On the other hand, $$\ip{b_1+\ldots+b_k}{f'} \leq
\norm{b_1+\ldots+b_k}\cdot\norm{f'}\leq\sqrt{k},$$ by
Cauchy-Schwartz and since $\norm{f'}\leq 1$. Combining the two
inequalities yields
$$k<\frac{4}{\alpha^2}.$$
\end{proof}

\subsection{More General $\set{0,1}^n$ Case}
\label{sec:more-general}
To prove Theorem~\ref{thm:general}, we first assume that the variables
$X_i$ are binary, albeit with skewed probabilities. Later we
reduce the general case to such variables.

We first prove a bound on the sum of effects of a subset of
players (in fact, a more general lemma is true, but we will not
use it).

\begin{lemma}
\label{lem:bound on effect} Let $n \in \N$ and
$f:\{0,1\}^n\mapsto [-1,1]$, and consider pairwise independent
binary random variables $X_1,\ldots,X_n$, with $\pr{X_i=0}=q$, for
all $i \in [n]$. Then for all $k \in [n]$,
$$\sum_{i \in [k]} \ef_i(f,X) \leq \sqrt{\frac{2k}{p}},$$
where $p=\min\{q,1-q\}$.
\end{lemma}

\begin{proof}
Let $z\in\{0,1\}$ be such that $\pr{X_i=z}=p$.
For each $i\in[k]$, fix $a_i\in\{-1,1\}$ such that
$$a_i\cdot \left(\E_X[f|X_i=z]-\E_X[f|X_i=1-z]\right) \geq 0.$$
Also define
the functions $b_i:\{0,1\}^k\mapsto [-1,1]$ as
$$b_i(x)=\left\{
\begin{array}{ll}
a_i & x_i=z\\
-a_i\cdot\frac{p}{1-p} & x_i=1-z
\end{array}\right.$$
for all $x = (x_1,\ldots,x_n) \in \set{0,1}^n$.
Again consider the quantity
\begin{align*}
\ip{b_i}{f} &= a_i\sum_{x:x_i=z}\Pr[X =x] f(x)-a_i\frac{p}{1-p}\sum_{x:x_i=1-z} \Pr[X = x]f(x)\\
&= a_i\pr{X_i=z}\cdot\E[f|X_i=z]-a_i\frac{p}{1-p}\cdot\pr{X_i=1-z}\cdot\E[f|X_i=1-z]\\
&= a_ip\left(\E[f|X_i=z]-\E[f|X_i=1-z]\right)\\
&= p \cdot \ef_i(f,X).
\end{align*}

By Cauchy-Schwartz and since $\norm{f} \leq 1$,
$$p\cdot\sum_{i \in [k]} \ef_i(f,X) = \ip{b_1+\ldots+b_k}{f} \leq \norm{b_1+\ldots+b_k}\cdot\norm{f}
\leq \norm{b_1+\ldots+b_k}.$$ We will now bound
$\norm{b_1+\ldots+b_k}$.

First we bound $\ip{b_i}{b_i}$.
\begin{align*}
\ip{b_i}{b_i} &= \pr{X_i=z} + \frac{p^2}{(1-p)^2}\pr{X_i=1-z}\\
&= p+\frac{p^2}{1-p}.
\end{align*}
Now we bound $\ip{b_i}{b_j}$ for $i\not=j$ using the pairwise independence of $X_i$ and $X_j$.
\begin{align*}
\ip{b_i}{b_j} &= a_ia_j\pr{X_i=X_j=z} + a_ia_j\frac{p^2}{(1-p)^2}\pr{X_i=X_j=1-z} -a_ia_j\frac{p}{1-p}\pr{X_i\not= X_j}\\
&= a_ia_j\left( p^2 + p^2 - 2p^2 \right)\\
&= 0.
\end{align*}
Thus,
$$\norm{b_1+\ldots+b_k}^2 \leq 2pk,$$
since $p \leq 1/2$.
The lemma follows.
\end{proof}

We use the previous lemma to get the following bound on the number of
players with large effect.

\begin{lemma}
\label{lem:binary} Let $n \in \N$ and
$f:\{0,1\}^n\mapsto [-1,1]$, and consider pairwise independent
binary random variables $X_1,\ldots,X_n$, with $\pr{X_i=0}=q$, for
all $i \in [n]$.
Then
for every positive $\alpha \in \R$,
$$K(f,X,\alpha)<\frac{2}{p\alpha^2},$$
where $p=\min\{q,1-q\}$.
\end{lemma}

\begin{proof}
As before, let $k=K(f,X,\alpha)$, and without loss of generality assume that the first $k$ variables are
the ones with effect $\alpha$.
Using Lemma~\ref{lem:bound on effect},
$$\alpha \cdot
k < \sum_{i\in [k]} \ef_i(f,X) \leq \sqrt{\frac{2k}{p}},$$ which implies
$$k < \frac{2}{p\alpha^2}.$$
\end{proof}

\subsection{Reducing the General Case to the Boolean Case}

We now wish to reduce the general case to the case in which the random variables are
binary and have skewed marginals.

\begin{lemma}
\label{lem:reduction} Let $n\in\N$, and let $S$ be a finite set.
Let $f:S^n\mapsto [-1,1]$ be some function. Let $X_1,\ldots,X_n$
be $n$ pairwise independent random variables taking values in $S$, and denote $X=(X_1,\ldots,X_n)$.
Let $\alpha>0$ and $0\leq p \leq 1$. Then there exist an
integer $k \in \N$, a function $g:\{0,1\}^k\mapsto [-1,1]$ and
pairwise independent binary random variables $Y_1,\ldots,Y_k$ such that
\begin{itemize}
\item For every $i \in [k]$, $\pr{Y_i=0} = \frac{p}{2}$, and
\item $K(f,X,p,\alpha)\leq 2\cdot K(g,Y,\alpha)=2k$,
\end{itemize}
where $Y=(Y_1,\ldots,Y_k)$.
\end{lemma}

\begin{proof}
Let $I\subseteq[n]$ be the set of $(p,\alpha)$-pivotal players in
$f$, and suppose that for at least $|I|/2$ of $i\in I$,
\begin{eqnarray} \label{eqn: a} \pr{\E[f|X_i]-\E[f]>\alpha}>\frac{p}{2}.\end{eqnarray} If this does not
hold, then simply consider the function $f'=1-f$, for which
(\ref{eqn: a}) will hold. Denote by $I^+$ the set of indices for
which (\ref{eqn: a}) holds. Without loss of generality, assume
$I^+=\{1,\ldots,k\}$, and note that $k\geq K(f,X,p,\alpha)/2$.

For every $i \in I^+$, denote
$$p_i = \Pr\left[\E[f|X_i]-\E[f]>\alpha\right]  > \frac{p}{2}.$$
Define $Y_i = Y_i(x_i)$ as follows:
\begin{itemize}
\item If $\E[f|X_i=x_i]-\E[f]>\alpha$, then with probability $\frac{p}{2p_i}$ set $Y_i = 0$
(independently of all other random variables).
\item Otherwise, set $Y_i = 1$.
\end{itemize}

Thus, for every $i\in I^+$, we have $\pr{Y_i=0} =
\frac{p}{2}$. Furthermore, since $X_1,\ldots,X_n$ are pairwise
independent, $Y_1,\ldots,Y_k$ are pairwise independent. All that remains is to define a function
$g$ in which all of the $Y_i$'s will have large effect.

To this end,
define the function $g:\set{0,1}^k\mapsto [-1,1]$ as
$$g(y)=\E\left[f\big|Y_1=y_1,\ldots,Y_k=y_k\right]$$
for all $y = (y_1,\ldots,y_k) \in \set{0,1}^k$.

Now, for every $i\in [k]$ and $z\in\set{0,1}$,
\begin{align}
\nonumber \E[f|Y_i=z] &= \sum_{y} \E\left[f|Y=y\right]\cdot\pr{Y=y|Y_i=z}\\
\nonumber &= \sum_{y} g(y)\cdot\pr{Y=y|Y_i=z}\\
\label{eqn: egyiz} &=\E[g|Y_i=z].
\end{align}
An additional claim we need is that $\E[f|Y_i=0]>\E[f]+\alpha$. Denote by
$$T=\set{t\in \supp(X_i):\E[f|X_i=t]>\E[f]+\alpha}.$$

For any $t \in T$,
\begin{align*}
\E[f|Y_i=0, X_i=t] &= \sum_{x:x_i=t} f(x)  \cdot \frac{p}{2 p_i} \cdot \frac{\Pr[X=x]}{\Pr[X_i=t, Y_i=0]} \\
%&= \sum_{x:x_i=t} f(x) \cdot \frac{\Pr[X=x]}{\Pr[X_i=t]} \cdot \frac{\Pr[X_i=t]}{\Pr[X_i=t, Y_i=0]} \cdot \frac{p}{2p_i} \\
&= \E[f|X_i=t] \cdot \frac{p}{2 p_i} \cdot \frac{1}{\Pr[Y_i=0|X_i=t]} \\
&= \E[f|X_i=t].
\end{align*}
Hence,
\begin{align*}
\E[f|Y_i=0] &= \sum_{t \in T} \E[f|X_i=t] \Pr[X_i=t|Y_i=0] > \E[f] + \alpha .
\end{align*}
Therefore, since
$$\E[f] = \E[f|Y_i=0] \Pr[Y_i=0] + \E[f|Y_i=1] \Pr[Y_i=1],$$
it follows that
$$\E[f|Y_i=1]<\E[f],$$ which implies
$$\E[f|Y_i=0] - \E[f|Y_i=1] > \alpha.$$

Thus, using (\ref{eqn: egyiz}), for all $i\in \{1,\ldots,k\}$, $$\E[g|Y_i=0]-\E[g|Y_i=1] =
\E[f|Y_i=0]-\E[f|Y_i=1] > \alpha.$$
\end{proof}

\subsection{Proof of Main Result}

We are now ready to prove Theorem~\ref{thm:general}.
\begin{proof}
By Lemma~\ref{lem:reduction}, there exists a function $g$ and
distribution $Y$ such that $K(f,X,p,\alpha)\leq 2\cdot
K(g,Y,\alpha)$. Since the distribution $Y$ is such that
$\pr{Y_i=0}=p/2$, and the $Y_i$'s are pairwise independent,
Lemma~\ref{lem:binary} implies that $K(g,Y,\alpha)<2/(p/2)\alpha^2 = 4/p\alpha^2$. Hence,
$$K(f,X,p,\alpha) < \frac{8}{p\alpha^2}.$$
\end{proof}

\subsection{The Effect of Sets of Players}

In this section we generalize Theorem~\ref{thm:general} and prove Theorem~\ref{thm:ave-case-sets}.
We first restate the theorem.
\begin{theorem}[Theorem~\ref{thm:ave-case-sets} Restated]
Fix some natural number $m$. Then for any set of $2m$-wise independent random variables
$X_1,\ldots,X_n$, any $0<\alpha<1$, $0<p<1$, and any function $f$, the following holds:
there exists a set $C \subseteq [n]$ of size $|C|\leq 8m / p\alpha^2$, such that
for all $T \subseteq [n]\setminus C$ of size $|T| \leq m$, the set $T$ is \textbf{not}
$(p,\alpha)$-pivotal.
\end{theorem}

\begin{proof}
We first bound the possible number of \textit{disjoint} pivotal sets.
Let $\{C_i\}_{i=1}^t$ be some maximal collection of sets (i.e.
$t$ is maximal) satisfying the following:
\begin{itemize}
\item $C_i\subseteq [n]$, $|C_i|\leq m$, for all $i\in [t]$.
\item For all $i\not= j$, $C_i\cap C_j = \emptyset$.
\item For all $i$, $C_i$ is $(p,\alpha)$-pivotal.
\end{itemize}
Since
$\{C_i\}_{i=1}^t$ is a maximal collection of such sets, any other $(p,\alpha)$-pivotal
set of size at most $m$ intersects at least one of the $C_i$'s.
We now provide an upper bound on the number $t$ of such sets.

To simplify the exposition, suppose all $C_i$'s are of size $m$,
$C_1 = \set{1,\ldots,m}$, $C_2=\set{m+1,\ldots,2m}$, and so on. Now consider the function
$$f'(X'_1,\ldots, X'_t, X_{mt+1},\ldots, X_n)= f(X_1,\ldots,X_n),$$
where $X'_i = X_{C_i}$. We call all players in $f'$ ``meta-players".
This function takes the same values
as $f$, except that it considers the inputs of all the players in $C_i$ as the input of one
meta-player. Note that $f'$ has the same expectation as $f$.

The variables $X_1,\ldots,X_n$ are $2m$-wise independent, and so
$X'_1 , \ldots , X'_t , X_{mt+1} , \ldots , X_n$ are pairwise independent. This holds because any
meta-player depends on at most $m$ original players, so any 2 meta-players consist of
at most $2m$ players that are all independent. Hence, every 2 meta-players are also independent.

By Theorem~\ref{thm:general}, the number of $(p,\alpha)$-pivotal players in $f'$ is less
than $8/p\alpha^2$. Note that if the set of players $C_i$ is $(p,\alpha)$-pivotal in $f$, then
the meta-player $X'_i$ is $(p,\alpha)$-pivotal in $f'$. Thus,
we can conclude that $t < 8/p\alpha^2$.

Let $C=\bigcup_{i\in [t]} C_i$. Now consider some set $T\subseteq [n]$ of size $|T|\leq m$.
If $T$ is disjoint from $C$, then
$T$ can \textbf{not} be $(p, \alpha)$-pivotal, since the $C_i$'s are a maximal collection
of disjoint pivotal sets.
%Conversely, if $T$ is $(p,\alpha)$-pivotal, then $T\cap C \not= \emptyset$.

We are now done: $|C|< 8m/p\alpha^2$ as required, and for
any $T \subseteq [n]\setminus C$ of size $|T| \leq m$, $T$ is \textbf{not}
$(p,\alpha)$-pivotal.
\end{proof}

\section{General KKL-Type Results}
\label{sec:kkl-type-results}

\subsection{Pairwise Independence -- Negative Result}
\label{scn: negative result for effect}

We will now present a balanced monotone function and a pairwise independent distribution such that
the effects of all players are 0. This will imply that there is no non-trivial lower bound on the
effect of a player for pairwise independent distributions, even for monotone functions.

The rough idea of the construction is that since the support of a pairwise independent distribution
can be small, monotonicity does not play much of a role (in the next section, however, we show that if the
support of the distribution is very small, some player must have large effect).
We begin by describing the distribution
$D$ used in the counter-example. $D$ will be the convex sum of two other pairwise independent
distributions $\mu$ and $\omu$.

Assume $n+1=2^k$, and identify the set $\{0,\ldots,n\}$ with $\{0,1\}^k$ (by the binary representation).
$\mu$ will be the uniform distribution over a set of $n+1$ strings in $\{0,1\}^{n}$.
%, where $z=(z_1,\ldots ,z_k)\in\{0,1\}^k$.
These $n+1$ elements in the support of $\mu$ will be denoted by $x^z$, where $z=(z_1,\ldots ,z_k)$ runs over all vectors in $\{0,1\}^k$.
Let $y=(y_1,\ldots,y_k)$ be a nonzero element of $\{0,1\}^k$, or, equivalently, an element of $\{1,\ldots,n\}$.
Then the $y$'th index of $x^z$ is
$$x^z_y = \ip{z}{y} \eqdef \sum_{i\in [k]} z_i\cdot y_i ~~\mathrm{mod}~2.$$
$\mu$ is the quintessential pairwise independent distribution, with marginals $1/2$.
Note that the support of $\mu$ consists
of the $(0,\ldots,0)$ vector and $n$ other vectors, each with $(n+1)/2$ ones and $(n-1)/2$ zeros.
Moreover, aside from the $(0,\ldots,0)$ vector,
all other vectors are incomparable (under the natural partial order on $\{0,1\}^n$).

$\omu$ will be the uniform distribution on vectors that complement those of $\mu$:
for every $x\in\supp(\mu)$, there is an $\overline{x}\in\supp(\omu)$ such that $\overline{x} = (1,\ldots,1)\oplus x$,
where $\oplus$ is the bit-wise XOR. Formally, $\omu$ is the uniform distribution over $n+1$ strings
$$\overline{x^z}\eqdef (1,\ldots,1) \oplus x^z,$$ where $z=(z_1,\ldots ,z_k)\in\{0,1\}^k$.
Note that the support of $\omu$ complements that of $\mu$: it consists
of the $(1,\ldots,1)$ vector and $n$ other vectors, each with $(n-1)/2$ ones and $(n+1)/2$ zeros.
Since $\mu$ is pairwise independent and has marginals $1/2$, so does $\omu$.

Set $$D=\frac{\mu}{2}+\frac{\omu}{2}.$$
Since $\mu$ and $\omu$ are pairwise independent and have marginals $1/2$, so does $D$.
Except for the $(1,\ldots,1)$ and $(0,\ldots,0)$ vectors, none of the vectors in the support
of $D$ are comparable (for $n \geq 7$). This means that every function $f$
with $f(1,\ldots,1)=1$ and $f(0,\ldots,0)=0$ is monotone on the support of $D$.

Define the function $f:\set{0,1}^n\mapsto\set{0,1}$ as follows. For all
$x\in\supp(\mu)$, $f(x)=0$. For all $x\in\supp(\omu)$, $f(x)=1$. Note that
regardless of how $f$ is defined on other inputs, $f$ is monotone on the
support of $D$. Furthermore, $f$ is balanced when the inputs are drawn from $D$.
Finally, it is possible to extend $f$ to all of $\set{0,1}^n$ in such a way that $f$ will remain
monotone and balanced.

It remains to show that all the effects of players in $f$ with respect
to $D$ are 0. Since $f$ is constant on the support of $\mu$, the effects of all players in $f$ with respect
to $\mu$ are 0. The same is true for $\omu$. Using Lemma \ref{lem: convex sum of dist} below, we conclude that the
effects of all players in $D$ are 0.

Note that for functions that are constant on some distribution,
all effects are trivially 0 with respect to that distribution.  Such functions, however, are not balanced.
The reason $f$ is interesting is that, with respect to $D$,
the effects are 0 despite $f$ being balanced.

\begin{lemma} \label{lem: convex sum of dist}
Let $\eta_1$ and $\eta_2$ be two distributions on
$(X_1,\ldots,X_n)\in\set{0,1}^n$ with marginals $0<p<1$,
and let $\eta=q\eta_1+(1-q)\eta_2$, where $0\leq q\leq 1$. Then for any $g:\set{0,1}^n\mapsto\set{0,1}$,
\begin{align*}
\E_\eta[g|X_i=1]&-\E_\eta[g|X_i=0]\\
&= q\left(\E_{\eta_1}[g|X_i=1]-\E_{\eta_1}[g|X_i=0]\right) +
(1-q)\left(\E_{\eta_2}[g|X_i=1]-\E_{\eta_2}[g|X_i=0]\right).
\end{align*}
\end{lemma}

\begin{proof}
\begin{align*}
\E_\eta[g|X_i=1]&-\E_\eta[g|X_i=0]\\
=&~ \frac{1}{p}\sum_{x:x_i=1} \eta(x)g(x) - \frac{1}{1-p}\sum_{x:x_i=0} \eta(x)g(x)\\
=&~ \frac{1}{p}\sum_{x:x_i=1} \Big(q\eta_1(x)+(1-q)\eta_2(x)\Big)g(x) -
\frac{1}{1-p}\sum_{x:x_i=0} \Big(q\eta_1(x)+(1-q)\eta_2(x)\Big)g(x)\\
=&~ \frac{1}{p}\sum_{x:x_i=1} q\eta_1(x)g(x) - \frac{1}{1-p}\sum_{x:x_i=0} q\eta_1(x)g(x)\\
&+\frac{1}{p}\sum_{x:x_i=1} (1-q)\eta_2(x)g(x) - \frac{1}{1-p}\sum_{x:x_i=0} (1-q)\eta_2(x)g(x)\\
=&~ q\left(\E_{\eta_1}[g|X_i=1]-\E_{\eta_1}[g|X_i=0]\right) +
(1-q)\left(\E_{\eta_2}[g|X_i=1]-\E_{\eta_2}[g|X_i=0]\right),
\end{align*}
where $x\in\set{0,1}^n$ for all sums above.
\end{proof}

\subsection{Pairwise Independence -- Negative Result for Influence}
\label{scn: negative result for influence}

The previous section deals with KKL-type theorems for effect.  In this section we ask whether a KKL-type theorem holds
for influence when the distribution is not fully independent.  We show that such a theorem does not hold; but
first, we recall the definition of influence. provide a precise definition of influence.

\begin{definition}[Definition~\ref{def:influence} restated]
Let $f:\{0,1\}^n \to \{0,1\}$ be a function, and let $\mu$ be a distribution on $\{0,1\}^n$.
The \emph{influence} of the $i$'th player is
$$ I_i(f,\mu) = \Pr_{x \sim \mu} \left[  f(x) \neq f(x \oplus e_i) \right] , $$
where $e_i$ is the vector with $1$ at the $i$'th index and $0$ elsewhere, and $\oplus$ is bitwise XOR.
\end{definition}

Note that the vector $x \oplus e_i$ may not even be in the support of $\mu$.  Thus, this may not be the
``correct'' measure of influence for non-independent distributions (which is one of the reasons for considering effect).
The original KKL theorem was proved in the case where $\mu$ is a fully-independent distribution (see \cite{KKL88}).
Understanding the most general scenario in which KKL holds is an interesting open question.  Our example from the
previous section shows that a KKL-type theorem for effect does not hold
under the assumption that $\mu$ is pairwise independent.  We now show that a KKL-type theorem does not hold for influence either, assuming only
pairwise independence (and monotonicity).

Consider the pairwise independent distribution $D$ from the previous section.  On the support of $D$, let $f$ be
defined as in the previous section.  Note that if $n$ is large enough, for any
$x,y \in \supp(D) \setminus \set{ (0,\ldots,0) , (1,\ldots,1)}$
and any $i,j \in [n]$,
we have that $x \oplus e_i$ is not comparable to either $y$ or $y \oplus e_j$.  Thus,
$f$ can be extended to a monotone function on all of $\{0,1\}^n$,
such that if $x \in \supp(D)$, then for any $i \in [n]$, $f(x \oplus e_i) = f(x)$.
This implies that $f$ is a balanced monotone function and $D$ is a pairwise independent distribution such that
$$ I_i(f,D) = 0 $$  for all $i \in [n]$.

\subsection{Pairwise Independence -- Positive Result}

In this section we prove our one positive result on KKL-type theorems for effect -- Theorem~\ref{thm: p.w. positive}.
We first restate the theorem.
\begin{theorem}[Theorem~\ref{thm: p.w. positive} restated]
Let $n +1 = 2^k$.
Let $\mu$ be a pairwise independent distribution on $(X_1,\ldots,X_n)\in\set{0,1}^n$, with
marginals $1/2$ and $|\supp(\mu)|=n+1$. Let $f:\set{0,1}^n\mapsto\set{0,1}$.
Then $$\sum_{i\in[n]}\left(\ef_i(f)\right)^2 = \frac{\Var[f]}{4}.$$
\end{theorem}

\begin{proof}
Since $\mu$ is pairwise independent and has support of size $n+1$,
$\mu$ is uniform on its support (see \cite{BGP07}).
Thus, for every $x \in \supp(\mu)$, we have $\mu(x) = 2^{-k}$.

We identify the set $\{0,\ldots,n\}$ with $\{0,1\}^k$ (by the binary representation).
There is a bijection between $\set{0,1}^k$ and the support of $\mu$:
for every $z \in \set{0,1}^k$ fix a corresponding $x^z \in \supp(\mu)$.

Let $y$ be a nonzero element of $\{0,1\}^k$, or, equivalently, an element of $\{1,\ldots,n\}$.
Denote by $\chi_y$ the map from $\set{0,1}^k$ to $\set{1,-1}$
defined as
$$\forall \ z \in \set{0,1}^k \ , \ \ \chi_y(z) = (-1)^{x^z[y]},$$
where $x^z[y]$ is the $y$'th index of $x^z$. Also denote
$\chi_{(0,\ldots,0)}$ the map defined as
$$\forall \ z \in \set{0,1}^k \ , \ \ \chi_{(0,\ldots,0)}(z) = 1.$$

We will consider the vector space of maps from $\set{0,1}^k$ to $\R$.
We will now show that the set of maps $\set{\chi_y}_{y \in \set{0,1}^k}$ form
an orthonormal basis for this vector space with respect to the inner product
$$\ip{g}{g'} = \sum_{z \in \set{0,1}^k} 2^{-k} g(z) g'(z).$$
For all $y \neq y'$ nonzero elements in $\set{0,1}^k$, we have
\begin{align*}
\ip{\chi_y}{\chi_{y'}} & = \sum_{z \in \set{0,1}^k} 2^{-k} \chi_{y}(z) \chi_{y'}(z) \\
& = \sum_{z \in \set{0,1}^k} 2^{-k} (-1)^{x^z[y]} (-1)^{x^z[y']} \\
& = \Pr[X_y= X_{y'}] - \Pr[X_y \neq X_{y'}]\\
& = 0,
\end{align*}
where the last equality follows since $\mu$ is pairwise independent with marginals $1/2$
(and we think of $y$ and $y'$ as elements of $[n]$).
Furthermore,
\begin{align*}
\ip{\chi_y}{\chi_{(0,\ldots,0)}} = \sum_{z \in \set{0,1}^k} 2^{-k} \chi_{y}(z)
= \Pr[X_y= 0] - \Pr[X_y =1]
= 0,
\end{align*}
where the last equality follows since the marginals are $1/2$.
Finally, for every $y \in \set{0,1}^k$, we have
\begin{align*}
\ip{\chi_y}{\chi_{y}} = \sum_{z \in \set{0,1}^k} 2^{-k} = 1. \end{align*}
Thus, the set of maps $\set{\chi_y}_{y \in \set{0,1}^k}$ form an orthonormal basis.

We can also think of $f$ as a map from $\set{0,1}^k$ to $\R$ as follows:
$$\forall \ z \in \set{0,1}^k \ , \ \ f(z) = f(x^z).$$
Thus, we can write
$$f = \sum_{y \in \set{0,1}^k} \widehat{f}(y) \chi_y,$$
where
$$\widehat{f}(y) = \ip{f}{\chi_y}$$
($\widehat{f}(\cdot)$ is called the Fourier transform of $f$).
By Parseval's equality,
$$\sum_{z \in \set{0,1}^k} \abs{f(z)}^2 = 2^k \sum_{y \in \set{0,1}^k} \abs{\widehat{f}(y)}^2.$$
We will now show that $\abs{\widehat{f}(y)}$ is twice the effect of the $y$'th player.
For a nonzero $y \in \set{0,1}^k$, since the marginals are $1/2$,
\begin{align*}
\widehat{f}(y)
&= \sum_{z \in \set{0,1}^k} 2^{-k} f(z) (-1)^{x^z[y]} \\
&= \sum_{z:x^z[y] = 0} 2^{-k} f(z) - \sum_{z:x^z[y] = 1} 2^{-k} f(z)  \\
&= 2 \left( \E[f|X_y = 0] - \E[f|X_y = 1] \right).
\end{align*}
In addition,
$$\widehat{f}\big((0,\ldots,0)\big)= \sum_{z \in \set{0,1}^k} 2^{-k} f(z) = \E[f].$$
Thus,
\begin{eqnarray*}
4 \sum_{i \in [n]} \abs{\ef_i(f)}^2 & = &  \sum_{y \in \set{0,1}^k} \abs{\widehat{f}(y)}^2 - \abs{\widehat{f}\big((0,\ldots,0) \big)}^2\\
& = & 2^{-k} \sum_{z \in \set{0,1}^k} \abs{f(z)}^2 - \abs{\widehat{f}\big((0,\ldots,0) \big)}^2 \\
& = & \E[f^2] - \left( \E[f] \right)^2\\
& = & \Var[f].
\end{eqnarray*}
\end{proof}

\end{document}